\documentclass[10pt]{amsart}
\usepackage{amssymb}
\usepackage[numbers,sort&compress]{natbib}
\renewcommand{\leq}{\leqslant}
\renewcommand{\geq}{\geqslant}
 \makeatletter \numberwithin{equation}{section}
\numberwithin{figure}{section} 
\theoremstyle{plain}
\newtheorem{thm}{Theorem}[section]

\newtheorem{lem}[thm]{Lemma} 
\theoremstyle{Definition}

\begin{document}
\title{ Stancu type generalization of
the $q-$Favard-Sz\`{a}sz operators }
\author{AL\.{I} Kara\.{I}SA,  DURHASAN TURGUT TOLLU and YAS{\.I}N ASAR}

\address[A. Karaisa]{DEPARTMENT OF MATHEMATICS-COMPUTER SCIENCE, FACULTY OF SCIENCES,
NECMETT\.{I}N ERBAKAN UNIVERSITY, MERAM YERLE\c{S}KES\.{I} , 42090
MERAM, KONYA, TURKEY} \email[A. Karaisa]{alikaraisa@hotmail.com,
akaraisa@konya.edu.tr}

\address[D.T.  Tollu]{DEPARTMENT OF MATHEMATICS-COMPUTER SCIENCE, FACULTY OF SCIENCES,
NECMETT\.{I}N ERBAKAN UNIVERSITY, MERAM YERLE\c{S}KES\.{I} , 42090
MERAM, KONYA, TURKEY} \email[D.T.  Tollu]{dttollu@konya.edu.tr}

\address[Y. Asar]{DEPARTMENT OF MATHEMATICS-COMPUTER SCIENCE, FACULTY OF SCIENCES,
NECMETT\.{I}N ERBAKAN UNIVERSITY, MERAM YERLE\c{S}KES\.{I} , 42090
MERAM, KONYA, TURKEY} \email[Y. Asar]{yasar@konya.edu.tr}

\begin{abstract}
In this paper, we introduce a Stancu type generalization of the
$q-$Favard-Sz\`{a}sz operators, estimate the rates of statistical
convergence and study the local approximation properties of these
operators.
\end{abstract}

\maketitle

\section{Introduction}
In \cite{8}, Jakimovski and Leviatan\ introduced a Favard Sz\`{a}sz
type operator, by using Appell polynomials $p_{k}\left( x\right) ,\
k\geq 0,$
defined by%
\begin{equation*}
g\left( u\right) e^{-ux}=\sum\limits_{k=0}^{\infty }p_{k}\left(
x\right) u^{k},
\end{equation*}
where $g\left( z\right) =\sum\limits_{n=0}^{\infty }a_{n}z^{n}$ is
an analytic function in the disc $\left\vert z\right\vert <R$, $R>1$ and $%
g\left( 1\right) \neq 0$, and they established several approximation
properties of these operators. Ciupa \cite{5}, by defined the following operators%
\begin{equation*}
P_{n,t}\left( f;x\right) =\frac{e^{-nx}}{g\left( 1\right) }%
\sum\limits_{k=0}^{\infty }p_{k}\left( nt\right) f\left( x+\frac{k}{n}%
\right)
\end{equation*}
and investigated the approximation properties and the rate of convergence of
these operators via the modulus of continuity.

In \cite{4}, Atakut and B\"{u}y\"{u}kyaz\i c\i\ studied some
approximation properties of the operators
\begin{equation*}
P_{n,t}^{\alpha ,\beta }\left( f;x\right) =\frac{e^{-nx}}{g\left( 1\right) }%
\sum\limits_{k=0}^{\infty }p_{k}\left( nt\right) f\left( x+\frac{k+\alpha }{%
n+\beta }\right)
\end{equation*}
where $p_{k}\left( nt\right) $ is an Appell polynomial in $nt$ which is a Stancu type generalization of the classical Favard-Sz\`{a}sz operators. Moreover, In \cite{3}, the same authors established the approximation properties of the operators
\begin{equation*}
L_{n}\left( f;q,x\right) =\frac{E^{-\left[ n\right] _{q}t}}{A\left(
1\right) }\sum\limits_{k=0}^{\infty }\frac{P_{k}\left( q;\left[
n\right] _{q}t\right) }{\left[ k\right] _{q}!}f\left( x+\frac{\left[
k\right] _{q}}{\left[ n\right] _{q}}\right) ,
\end{equation*}
which is a $q-$analogue of the classical Favard-Sz\`{a}sz operators
related to the $q-$Appell polinomials. They also estimated the rate of convergence of these operators.

Now, let us define Stancu type generalization of the $q-$Favard-Sz\`{a}sz
operators as follows:
\begin{equation}
T_{n,t}^{\alpha ,\beta }\left( f;q;x\right) =\frac{E_{q}^{-\left[
n\right]
_{q}t}}{A\left( 1\right) }\sum_{k=0}^{\infty }\frac{P_{k}\left( q;\left[ n%
\right] _{q}t\right) }{\left[ k\right] _{q}!}f\left( x+\frac{\left[
k\right] _{q}+\alpha }{\left[ n\right] _{q}+\beta }\right),
\label{A}
\end{equation}
where $\{ P_{k}( q;.)\}_{k\geq 0}$ is a $q-$Appell polynomial set
which is generated by
\begin{equation}\label{1}
A\left( u\right) e_{q}^{\left[ n\right] _{q}tu}=\sum_{k=0}^{\infty }\frac{%
P_{k}\left( q;\left[ n\right] _{q}t\right) u^{k}}{\left[ k\right]
_{q}!}
\end{equation} and $A(t)$ is defined by
\begin{eqnarray*}
A(u)=\sum_{k=0}^{\infty }a_{k}u^{k}.
\end{eqnarray*}

In this work, we investigate a Korovkin theorem and the rate of statistical
convergence by using modulus of continuity of (\ref{A}). We also
obtain some local approximation results of these new operators.

Let us recall some definitions and notations regarding the concept of
$q-$calculus. Further results can be found in \cite{kitap 1}. In the
sequel, $q$ is a real
number satisfying $0<q<1$. For $n\in\mathbb{N}$, the $q-$integer $\left[ n\right] _{q}$ is defined by%
\begin{eqnarray*}
 \left[ n\right] _{q}:=\frac{1-q^{n}}{1-q}
\end{eqnarray*}
and the $q-$factorial $\left[ n\right] _{q}!$ is defined as following%
\begin{eqnarray*}
\left[ n\right] _{q}!:=\left\{
\begin{tabular}{ll}
$\left[ n\right] _{q}\left[ n-1\right] _{q}\cdots \left[ 1\right] _{q},$ & $%
n\in\mathbb{N}$ \\
$1,$ & $n=0$%
\end{tabular}%
\right.
\end{eqnarray*}
The $q-$binomial coefficients are given by%
\begin{eqnarray*}
\left[
\begin{array}{c}
n \\
k
\end{array}
\right] _{q}=\frac{\left[ n\right] _{q}!}{\left[ k\right]
_{q}!\left[ n-k \right] _{q}!},0\leq k\leq n.
\end{eqnarray*}
The $q-$derivative $D_{q}f$ of a function $f$ is defined by
\begin{eqnarray*}
\left( D_{q}f\right) \left( x\right) =\frac{f\left( x\right)
-f\left( qx\right) }{\left( 1-q\right) x}, \ x\neq 0.
\end{eqnarray*}
Also, if there exists $\frac{df}{dx}\left( 0\right) $, then $%
\left( D_{q}f\right) \left( 0\right) =\frac{df}{dx}\left( 0\right)
$. The following q-derivatives of the product of the functions $f(x)$
and $g(x)$ are equivalent:
\begin{eqnarray*}
D_{q}\left(f(x)g(x)\right) =f\left( qx\right)
D_{q}g\left(x\right)+g(x)D_{q}f\left(x\right)
\end{eqnarray*}
and
\begin{eqnarray*}
D_{q}\left(f(x)g(x)\right) =f\left( x\right)
D_{q}g\left(x\right)+g(qx)D_{q}f\left(x\right).
\end{eqnarray*}
The
$q-$analogues of the exponential function are given by%
\begin{eqnarray*}
e_{q}^{x}=\sum\limits_{n=0}^{\infty }\frac{x^{n}}{\left[ n\right]
_{q}!}
\end{eqnarray*}
and
\begin{eqnarray*}
E_{q}^{x}=\sum\limits_{n=0}^{\infty }q^{\frac{n\left( n-1\right) }{2}}\frac{%
x^{n}}{\left[ n\right] _{q}!}.
\end{eqnarray*}
The exponential functions have the following properties:%
\begin{eqnarray*}
D_{q}\left( e_{q}^{ax}\right) =ae_{q}^{ax},\ D_{q}\left(
E_{q}^{ax}\right) =aE_{q}^{aqx},\
e_{q}^{x}E_{q}^{-x}=E_{q}^{x}e_{q}^{-x}=1.
\end{eqnarray*}

\section{statistical approximation properties}

Before proceeding further, let us give basic definition and notation
on the concept of the statistical convergence which was introduced
by Fast \cite{9}. Let $K$ be a subset of $\mathbb{N}$, the set of
natural numbers. Then, $K_{n}=\left\{k\leq n: k\in K \right\}$. The
natural density of $K$ is defined by
$\delta(K)=\lim_{n}\frac{1}{n}|K_{n}|$ provided that the limit exists, where
$|K_{n}|$ denotes the cardinality of the set $K_{n}$. A sequence
$x=(x_{k})$ is called statistically convergent to
the number $\ell \in\mathbb{R}$, denoted by $st-\lim x=\ell$. For each $\epsilon >0,$
the set
$K_{\varepsilon}=\left\{k\in\mathbb{N}:|x_{k}-\ell|\geq\epsilon
\right\} $ has a natural density zero, that is

\begin{eqnarray*}
\lim_{n \to \infty}\frac{1}{n}\left|\{k\leq n: |x_{k}-\ell|\geq
\epsilon \}\right|=0.
\end{eqnarray*}
 It is well know that every statistically convergence sequence is ordinary convergent, but the
converse is not true.

The concept of statistical convergence was firstly used in approximation
theory by Gadjiev and Orhan \cite{10}. They proved the
Bohman$-$Korovkin type approximation theorem for statistical
convergence. For further information related to the statistical
approximation of the operators, the followings are remarkable among
others: \cite{11, 12, 13, 14, 15}.

Now, we may begin the following lemma which is needed proving our
main result.
\begin{lem}\label{lemma3.1}
For $n\in \mathbb{N}$, $x\in [0,\infty)$ and $0<q<1$, we have%
\begin{eqnarray}
T_{n,t}^{\alpha ,\beta }\left( e_{0}(t);q;x\right)& =&1,\label{3}\\
 T_{n,t}^{\alpha ,\beta }\left( e_{1}(t);q;x\right)
&=&x+\frac{\left[n\right]_{q}t}{\left[
n\right]_{q}+\beta}+\frac{\alpha }{\left[ n\right] _{q}+\beta
}+\frac{RD_{q}(A(1))}{\left[ n\right]_{q}+\beta},\label{4}\\
T_{n,t}^{\alpha ,\beta }\left( e_{2}(t);q;x\right)& =&\left(x+\frac{\left[ n%
\right] _{q}t}{\left[ n\right] _{q}+\beta }\right)^{2}+\frac{2x[\alpha+RD_{q}(A(1))]}{%
\left[ n\right] _{q}+\beta }\nonumber\\
&&+\frac{\alpha^{2}+2R\alpha D_{q}(A(1)+D^{2}_{q}(A(1))}{\left(
\left[ n\right] _{q}+\beta \right) ^{2}} +\frac{ \left[ n\right]
_{q}t[2\alpha+D_{q}(A(1))]}{\left( \left[ n\right] _{q}+\beta
\right) ^{2}}+\frac{ q\left[ n\right]_{q}t}{\left( \left[ n\right]
_{q}+\beta \right) ^{2}}, \label{5}
\end{eqnarray} where
$R=\frac{ e_{q}^{q\left[ n\right] _{q}t}E_{q}^{-\left[ n%
\right] _{q}t}}{ A\left( 1\right)}$ ,
 and $e_v(t)=(x+t)^v$, $v=0, 1, 2$.
\end{lem}

\begin{proof}
If we consider (\ref{1}), we can show the following
\begin{equation}
\sum_{k=0}^{\infty }\frac{p_{k}\left( q;\left[ n\right] _{q}t\right) }{\left[
k\right] _{q}!}=A\left( 1\right) e_{q}^{\left[ n\right] _{q}t},  \label{6}
\end{equation}

\begin{equation}  \label{7}
\sum_{k=0}^{\infty }\frac{p_{k}\left( q;\left[ n\right] _{q}t\right) }{\left[
k\right] _{q}!}\left[ k\right] _{q}=A\left( 1\right) \left[ n\right]
_{q}te_{q}^{\left[ n\right] _{q}t}+e_{q}^{q\left[ n\right]
_{q}t}D_{q}A\left( 1\right)
\end{equation}
and
\begin{equation}  \label{8}
\sum_{k=0}^{\infty }\frac{p_{k}\left( q;\left[ n\right] _{q}t\right)
}{\left[ k\right] _{q}!}\left[ k\right] _{q}^{2}=D_{q}^{2}\left(
A\left( 1\right) \right) e_{q}^{\left[ n\right] _{q}t}+D_{q}(A\left(
1\right)) \left[ n\right] _{q}te_{q}^{\left[ n\right] _{q}t}+\left[
n\right] _{q}tD_{q}A\left( q\right) e_{q}^{q\left[ n\right]
_{q}t}+A(1) \left[ n\right]^{2} _{q}t^{2}e_{q}^{\left[ n\right]
_{q}t}.
\end{equation}

By using the relations (\ref{6})-(\ref{8}), from (\ref{A}), we
obtain the results

\begin{equation*}
T_{n,t}^{\alpha ,\beta }\left( e_{0}(t);q;x\right) =\frac{E_{q}^{-\left[ n%
\right] _{q}t}}{A\left( 1\right) }\sum_{k=0}^{\infty }\frac{p_{k}\left( q;%
\left[ n\right] _{q}t\right) }{\left[ k\right] _{q}!}=\frac{E_{q}^{-\left[ n%
\right] _{q}t}}{A\left( 1\right) }A\left( 1\right) e_{q}^{\left[ n\right]
_{q}t}=1,
\end{equation*}%
\begin{eqnarray*}
T_{n,t}^{\alpha ,\beta }\left( e_{1}(t);q;x\right) &=&\frac{E_{q}^{-\left[ n%
\right] _{q}t}}{A\left( 1\right) }\sum_{k=0}^{\infty }\frac{p_{k}\left( q;%
\left[ n\right] _{q}t\right) }{\left[ k\right] _{q}!}\left( x+\frac{\left[ k%
\right] _{q}+\alpha }{\left[ n\right] _{q}+\beta }\right) \\
&=&x+\frac{\alpha E_{q}^{-\left[ n\right] _{q}t}}{A\left( 1\right) \left( %
\left[ n\right] _{q}+\beta \right) }\sum_{k=0}^{\infty }\frac{p_{k}\left( q;%
\left[ n\right] _{q}t\right) }{\left[ k\right] _{q}!}+\frac{E_{q}^{-\left[ n%
\right] _{q}t}}{A\left( 1\right) \left( \left[ n\right] _{q}+\beta \right) }%
\sum_{k=0}^{\infty }\frac{p_{k}\left( q;\left[ n\right] _{q}t\right) }{\left[
k\right] _{q}!}\left[ k\right] _{q} \\
&=&x+\frac{\alpha }{\left[ n\right] _{q}+\beta }+\frac{E_{q}^{-\left[ n%
\right] _{q}t}}{A\left( 1\right) \left( \left[ n\right] _{q}+\beta \right) }%
\left( A\left( 1\right) \left[ n\right] _{q}te_{q}^{\left[ n\right]
_{q}t}+e_{q}^{q\left[ n\right] _{q}t}D_{q}A\left( 1\right) \right) \\
&=&x+\frac{\left[n\right]_{q}t}{\left[
n\right]_{q}+\beta}+\frac{\alpha }{\left[ n\right] _{q}+\beta
}+\frac{D_{q}(A(1))R}{\left[ n\right]_{q}+\beta}
\end{eqnarray*}

and
\begin{eqnarray*}
T_{n,t}^{\alpha ,\beta }\left( e_{2}(t);q;x\right)  &=&\frac{E_{q}^{-\left[ n%
\right] _{q}t}}{A\left( 1\right) }\sum_{k=0}^{\infty }\frac{p_{k}\left( q;%
\left[ n\right] _{q}t\right) }{\left[ k\right] _{q}!}\left( x+\frac{\left[ k%
\right] _{q}+\alpha }{\left[ n\right] _{q}+\beta }\right) ^{2} \\
&=&x^{2}+\frac{2\alpha x}{\left[ n\right] _{q}+\beta }+\frac{\alpha ^{2}}{%
\left( \left[ n\right] _{q}+\beta \right) ^{2}}+\frac{2xE_{q}^{-\left[ n%
\right] _{q}t}}{A\left( 1\right) \left( \left[ n\right] _{q}+\beta \right) }%
\sum_{k=0}^{\infty }\frac{p_{k}\left( q;\left[ n\right] _{q}t\right) }{\left[
k\right] _{q}!}\left[ k\right] _{q} \\
&&+\frac{2\alpha E_{q}^{-\left[ n\right] _{q}t}}{A\left( 1\right) \left( %
\left[ n\right] _{q}+\beta \right) ^{2}}\sum_{k=0}^{\infty }\frac{%
p_{k}\left( q;\left[ n\right] _{q}t\right) }{\left[ k\right] _{q}!}\left[ k%
\right] _{q} \\
&&+\frac{E_{q}^{-\left[ n\right] _{q}t}}{A\left( 1\right) \left( \left[ n%
\right] _{q}+\beta \right) ^{2}}\sum_{k=0}^{\infty }\frac{p_{k}\left( q;%
\left[ n\right] _{q}t\right) }{\left[ k\right] _{q}!}\left[ k\right] _{q}^{2}
\\
&=& \left(x+\frac{\left[ n%
\right] _{q}t}{\left[ n\right] _{q}+\beta }\right)^{2}+\frac{2x[\alpha+RD_{q}(A(1))]}{%
\left[ n\right] _{q}+\beta }\nonumber\\
&&+\frac{\alpha^{2}+2R\alpha D_{q}(A(1)+D^{2}_{q}(A(1))}{\left(
\left[ n\right] _{q}+\beta \right) ^{2}} +\frac{ \left[ n\right]
_{q}t[2\alpha+D_{q}(A(1))]}{\left( \left[ n\right] _{q}+\beta
\right) ^{2}}+\frac{ q\left[ n\right]_{q}t}{\left( \left[ n\right]
_{q}+\beta \right) ^{2}}.
\end{eqnarray*}%
Hence, the proof is completed.
\end{proof}

\begin{thm}
Assume that $q:=(q_n)$, $0<q_n<1$ be a sequence satisfying the following conditions:
\begin{eqnarray}\label{y1}
st-\lim_n q_n=1,\ \  st-\lim_n q_n^n=b,\ \  b<1
\end{eqnarray}
Then, if $f$ is any monotone increasing continuous function defined on $[0,a]$, we have the following:
\begin{equation*}\label{y2}
st-\lim_n \parallel T_{n,t}^{\alpha,\beta}(f,q_n;.)-f
\parallel_{C[0,a]}=0.
\end{equation*}

\end{thm}

\begin{proof}
It is enough to prove that
\begin{equation*}\label{y3}
st-\lim_n \parallel
T_{n,t}^{\alpha,\beta}(e_v(t),q_n;.)-e_v(t)\parallel_{C[0,a]}=0
\end{equation*}
where $v=0,1,2.$

From the equation (\ref{3}), it is easy to obtain that
\begin{equation*}\label{y4}
st-\lim_n \parallel
T_{n,t}^{\alpha,\beta}(e_0(t),q_n;.)-e_0(t)\parallel_{C[0,a]}=0.
\end{equation*}

By taking $\sup_{x,t\in [0,a]}$ in (\ref{4}), we get
\begin{equation*}\label{y5}
\parallel T_{n,t}^{\alpha,\beta}(e_1(t),q_n;.)-e_1(t)\parallel_{C[0,a]}\leq\frac{\alpha+RD_{q}(A(1))}{[n]_q+\beta}+\frac{\beta a}{[n]_q+\beta}.
\end{equation*}

Now, let $\epsilon>0$ be given, define the following sets:
\begin{eqnarray*}\label{y6}
K&:=& \left\lbrace  k: \parallel
T_{n,t}^{\alpha,\beta}(e_1(t),q_k;.)-e_1(t) \parallel \geq \epsilon
\right\rbrace,\\ K_1&:=& \left\lbrace  k: \frac{\beta
a}{[n]_{q_{k}}+\beta} \geq \frac{\epsilon}{2} \right\rbrace ,\\
K_2&:=& \left\lbrace  k: \frac{R D_{q}(A(1))+ \alpha
}{[n]_{q_{k}}+\beta}  \geq \frac{\epsilon}{2} \right\rbrace ,
\end{eqnarray*} such that $K\subseteq K_1 \cup K_2$.

From (\ref{y1}), one can see that
\begin{equation*} \label{y9}
\delta \left\lbrace k \leq n : \parallel
T_{n,t}^{\alpha,\beta}(e_1(t),q_n;.)-e_1(t)\parallel_{C[0,a]} \geq
\epsilon \right\rbrace \leq \delta \left\lbrace k \leq n:
\frac{\beta a}{[n]_{q_{k}}+\beta} \geq \frac{\epsilon}{2}
\right\rbrace + \delta \left\lbrace k \leq n: \frac{R + \alpha
}{[n]_{q_{k}}+\beta} \geq \frac{\epsilon}{2} \right\rbrace.
\end{equation*}
By (\ref{y1}), it is obvious that,
\begin{equation}\label{y99}
st-\lim_n\left( \frac{1}{[n]_{q_{n}}+\beta}\right) =0.
\end{equation}
Thus, we have
\begin{equation} \label{y10}
st-\lim_n \parallel
T_{n,t}^{\alpha,\beta}(e_1(t),q_n;.)-e_1\parallel_{C[0,a]}=0.
\end{equation}

By taking $\sup_{x,t\in [0,a]}$ in (\ref{5}), one can write the
following
\begin{eqnarray*} \label{y11}
\parallel T_{n,t}^{\alpha,\beta}(e_2(t),q_n;.)-e_2(t)\parallel_{C[0,a]}& \leq& \frac{\beta^{2}a^{2}\left[n\right] _{q}}{\left[ n\right] _{q}+\beta }+\frac{2a[\alpha+RD_{q}(A(1))]}{%
\left[ n\right] _{q}+\beta } +\frac{\alpha^{2}+2R\alpha
D_{q}(A(1))+D^{2}_{q}(A(1))}{\left( \left[ n\right] _{q}+\beta
\right) ^{2}}\nonumber\\&& +\frac{ \left[ n\right]
_{q}a[2\alpha+D_{q}(A(1))]}{\left( \left[ n\right] _{q}+\beta
\right) ^{2}}+\frac{ q\left[ n\right]_{q}a}{\left( \left[ n\right]
_{q}+\beta \right) ^{2}}.
\end{eqnarray*}

It is obvious that
\begin{equation*} \label{y11}
st-\lim_n \left( \frac{1}{([n]_{q_{n}}+\beta)^2} \right) = 0,\
st-\lim_n \left( \frac{[n]_{q_{n}}}{([n]_{q_{n}}+\beta)^2} \right) =
0\,\,\textrm{and }\,\,st-\lim_n \left(
\frac{q_{n}[n]_{q_{n}}}{([n]_{q_{n}}+\beta)^2} \right) =0.
\end{equation*}

Now, let $\epsilon>0$ be given, we define the following sets:
\begin{eqnarray*}\label{y12}
V&:=& \left\lbrace  k: \parallel
T_{n,t}^{\alpha,\beta}(e_2(t),q_k;.)-e_2(t) \parallel \geq \epsilon
\right\rbrace ,\\
V_1&:=& \left\lbrace  k: \frac{2a\left[\alpha+ RD_{q}(A(1))
\right]}{[n]_{q_{k}}+\beta}  \geq \frac{\epsilon}{^4}
\right\rbrace ,\\
V_2&:=& \left\lbrace  k: \frac{\alpha^{2}+2R\alpha
D_{q}((A(1))+D^{2}_{q}(A(1))+\beta^2
a^2}{([n]_{q_{k}}+\beta)^2} \geq \frac{\epsilon}{^4} \right\rbrace,\\
V_3&:=& \left\lbrace  k: \frac{
a[n]_{q_{k}}[2\alpha+D_{q}(A(1))]}{([n]_{q_{k}}+\beta)^2} \geq
\frac{\epsilon}{^4} \right\rbrace,\\
V_4&:=& \left\lbrace  k: \frac{
aq_{k}[n]_{q_{k}}}{([n]_{q_{k}}+\beta)^2} \geq \frac{\epsilon}{^4}
\right\rbrace
\end{eqnarray*} 
such that $V\subseteq V_1 \cup V_2 \cup V_3\cup V_4$.

Thus, we obtain
\begin{eqnarray}\label{y15}
\delta \left\lbrace k \leq n : \parallel
T_{n,t}^{\alpha,\beta}(e_2(t),q_n;.)-e_2(t)\parallel_{C[0,a]} \geq
\epsilon \right\rbrace &\leq & \delta \left\lbrace k \leq n:
 \frac{2a\left[\alpha+ D_{q}(A(1))
\right]}{[n]_{q_{k}}+\beta}  \geq \frac{\epsilon}{^4} \right\rbrace
\nonumber\\  &+& \delta \left\lbrace k \leq n:
\frac{\alpha^{2}+2R\alpha D_{q}((A(1))+D^{2}_{q}(A(1))+\beta^2
a^2}{([n]_{q_{k}}+\beta)^2} \geq \frac{\epsilon}{^4} \right\rbrace \nonumber\\
&+& \delta \left\lbrace k \leq n: \frac{
a[n]_{q_{k}}[2\alpha+D_{q}(A(1))]}{([n]_{q_{k}}+\beta)^2} \geq
\frac{\epsilon}{^4} \right\rbrace\nonumber\\
&+& \delta \left\lbrace k \leq n: \frac{
aq_{k}[n]_{q_{k}}}{([n]_{q_{k}}+\beta)^2} \geq \frac{\epsilon}{^4}
\right\rbrace.
\end{eqnarray}

Hence, (\ref{y99}), (\ref{y11}) and (\ref{y15}) imply that
\begin{equation}\label{y16}
st-\lim_n \parallel
T_{n,t}^{\alpha,\beta}(e_2(t),q_n;.)-e_2(t)\parallel_{C[0,a]}=0.
\end{equation}
\end{proof}

\section{Rates of statistical convergence}
In this section, we give the rates of statistical convergence of the
operators $T_{n,t}^{\alpha ,\beta }\left( f;q;x\right)$ by means of
modulus of continuity with the help of functions from Lipschitz class.

The modulus of continuity of $f$, $\omega(f,\delta)$ is defined by
\begin{equation*}
\omega(f,\delta)=\sup_{\substack{|x-y|\leq\delta \\ x,y\in [0,a]
}}|f(x)-f(y)|.
\end{equation*}
It is well-known that for a function $f \in C[0,a]$,
\begin{equation*}
\lim_{n\to 0^{+}}\omega(f,\delta)=0
\end{equation*}
for any $\delta>0$
\begin{equation}\label{r1}
|f(x)-f(y)|\leq \omega(f,\delta)\left(\frac{|x-y|}{\delta}+1\right).
\end{equation}
Now, we prove the following theorem for the rate of pointwise
convergence of the operators $T_{n,t}^{\alpha ,\beta }\left(
f;q;x\right)$ to the function $f(x+t)$ by means of modulus of
continuity.
\begin{thm}\label{thm4.1}
If the sequence $q:=(q_{n})$ satisfies the condition (\ref{y1}),
$x\in [0,\infty)$ and $t\geq 0$, then we have
\begin{equation*}\label{rr1}
|T_{n,t}^{\alpha ,\beta }\left( f;q;x\right)-f(x)|\leq
2\omega(f,\sqrt{\delta_{n,t}}),
\end{equation*}
for all  $f\in C^{\ast}[0,\infty)$, where
$\delta_{n,t}=T_{n,t}^{\alpha ,\beta }\left(
(s-e_{1}(t))^{2};q_{n};x\right)$

\end{thm}
\begin{proof}
To prove the theorem, we will use the linearity and positivity of the operators
$T_{n,t}^{\alpha ,\beta }\left( f;q;x\right)$. By (\ref{rr1}), we have
\begin{eqnarray*}
|T_{n,t}^{\alpha ,\beta }\left( f;q;x\right)-f(x+t)| &\leq&\frac{E_{q}^{-\left[ n%
\right] _{q}t}}{A\left( 1\right) }\sum_{k=0}^{\infty }\frac{p_{k}\left( q;%
\left[ n\right] _{q}t\right) }{\left[ k\right] _{q}!}\left |f\left( x+\frac{\left[ k%
\right] _{q}+\alpha }{\left[ n\right] _{q}+\beta }\right)-f(x+t)\right | \\
&\leq& \frac{E_{q}^{-\left[ n%
\right] _{q}t}}{A\left( 1\right) }\sum_{k=0}^{\infty }\frac{p_{k}\left( q;%
\left[ n\right] _{q}t\right) }{\left[ k\right] _{q}!}\omega(f,\delta)\left\{\frac{1}{\delta}\left | \frac{\left[ k%
\right] _{q}+\alpha }{\left[ n\right] _{q}+\beta }-t\right |+1\right\}\\
&=&\left\{\frac{1}{\delta}\frac{E_{q}^{-\left[ n%
\right] _{q}t}}{A\left( 1\right) }\sum_{k=0}^{\infty }\frac{p_{k}\left( q;%
\left[ n\right] _{q}t\right) }{\left[ k\right] _{q}!}\left | \frac{\left[ k%
\right] _{q}+\alpha }{\left[ n\right] _{q}+\beta }-t\right
|+1\right\}\omega(f,\delta) .
\end{eqnarray*}
If we apply Cauchy-Schwarz inequality for sums, we obtain
\begin{eqnarray*}
\sum_{k=0}^{\infty }\frac{p_{k}\left( q;%
\left[ n\right] _{q}t\right) }{\left[ k\right] _{q}!}\left | \frac{\left[ k%
\right] _{q}+\alpha }{\left[ n\right] _{q}+\beta
}-t\right|^{2}\leq\left(\sum_{k=0}^{\infty }\frac{p_{k}\left( q;%
\left[ n\right] _{q}t\right) }{\left[ k\right] _{q}!} \right)^{1/2}
\left(\sum_{k=0}^{\infty }\frac{p_{k}\left( q;%
\left[ n\right] _{q}t\right) }{\left[ k\right] _{q}!}\left ( \frac{\left[ k%
\right] _{q}+\alpha }{\left[ n\right] _{q}+\beta
}-t\right)^{2}\right)^{1/2}.
\end{eqnarray*}
Using above inequality and by (\ref{3}), we have
\begin{eqnarray*}
|T_{n,t}^{\alpha ,\beta }\left( f;q;x\right)-f(x+t)|
&\leq&\omega(f,\delta)\left\{1+\frac{1}{\delta}\left[\frac{E_{q}^{-\left[ n%
\right] _{q}t}}{A\left( 1\right) }
\sum_{k=0}^{\infty }\frac{p_{k}\left( q;%
\left[ n\right] _{q}t\right) }{\left[ k\right] _{q}!}\left ( \frac{\left[ k%
\right] _{q}+\alpha }{\left[ n\right] _{q}+\beta
}-t\right)^{2}\right]^{1/2}\right\}\\
&=&\omega(f,\delta)\left\{1+\frac{1}{\delta}
\left[\frac{E_{q}^{-\left[ n%
\right] _{q}t}}{A\left( 1\right) }\sum_{k=0}^{\infty }\frac{p_{k}\left( q;%
\left[ n\right] _{q}t\right) }{\left[ k\right] _{q}!}\left ( \frac{\left[ k%
\right] _{q}+\alpha }{\left[ n\right] _{q}+\beta
}-t\right)^{2}\right]^{1/2}\right\}\\
 &=&\omega(f,\delta)\left\{1+\frac{1}{\delta}\left[T_{n,t}^{\alpha ,\beta }\left( s-e_{1}(t))^{2};q;x\right)\right]^{1/2}\right\}
\end{eqnarray*}
such that we choose 
$\delta=\sqrt{\delta_{n,t}}=T_{n,t}^{\alpha ,\beta }\left(
(s-e_{1}(t))^{2};q_{n};x\right)$. This step concludes the proof.
\end{proof}
\begin{thm}
If the sequence $q:=(q_{n})$ satisfies the condition (\ref{y1}) and
$f\in C[0,a]$, then we have
\begin{equation*}\label{r1}
\parallel T_{n,t}^{\alpha,\beta}(f;q_n;.)-f(.)\parallel_{C[0,a]}\leq 2\omega(f,\sqrt{\delta_{n}}),
\end{equation*} where

\begin{equation*}\label{a0}
\delta_{n}= \frac{\beta^{2}t^{2}\left[n\right] _{q}}{\left( \left[
n\right] _{q}+\beta \right) ^{2} }+\frac{\alpha^{2}+2R\alpha
D_{q}(A(1))+D^{2}_{q}(A(1))}{\left( \left[ n\right] _{q}+\beta
\right) ^{2}} +\frac{ \left[ n\right]
_{q}t[2\alpha+D_{q}(A(1))]}{\left( \left[ n\right] _{q}+\beta
\right) ^{2}}+\frac{ q\left[ n\right]_{q}t}{\left( \left[ n\right]
_{q}+\beta \right) ^{2}}.
\end{equation*}
\end{thm}
\begin{proof}
Now, let us estimate the second moment of the operators
$T_{n,t}^{\alpha ,\beta }\left( f;q;x\right)$. From
(\ref{3})-(\ref{5}), we get
\begin{eqnarray}\label{a1}
T_{n,t}^{\alpha ,\beta }\left( (s-e_{1}(t))^{2};q;x\right)&=&
\frac{t^{2}\left[n\right]^{2} _{q}}{\left( \left[ n\right]
_{q}+\beta \right) ^{2} }+\frac{\alpha^{2}+2R\alpha
D_{q}(A(1))+D^{2}_{q}(A(1))}{\left( \left[ n\right] _{q}+\beta
\right) ^{2}}\\ &&+\frac{ \left[ n\right]
_{q}t[2\alpha+D_{q}(A(1))]}{\left( \left[ n\right] _{q}+\beta
\right) ^{2}}+\frac{ q\left[ n\right]_{q}t}{\left( \left[ n\right]
_{q}+\beta \right) ^{2}}.\nonumber
\end{eqnarray}
By Theorem \ref{thm4.1}, we have
\begin{equation}\label{a2}
|T_{n,t}^{\alpha ,\beta }\left(
f;q;x\right)-f(x+t)|\leq\omega(f,\delta)\left\{1+\frac{1}{\delta}\left[T_{n,t}^{\alpha
,\beta }\left( s-e_{1}(t))^{2};q;x\right)\right]^{1/2}\right\}
\end{equation}
Substituting (\ref{a1}) into (\ref{a2}) and letting $\delta=\delta_{n}$ in (\ref{a2}), we obtain
\begin{equation*}\label{r1}
\parallel T_{n,t}^{\alpha,\beta}(f;q_n;.)-f(.)\parallel_{C[0,a]}\leq
2\omega(f,\sqrt{\delta_{n}}).
\end{equation*}
Thus, we get the desired result.
\end{proof}
Now, we give the rate of convergence of the operators $T_{n,t}^{\alpha
,\beta}$  with the help of functions from Lipschitz  class
$Lip_{M}(\alpha)$ where $M>0$ and $0<\alpha\leq 1$. A function $f$ is an
element of $Lip_{M}(\alpha)$ if

\begin{equation}\label{a3}
\left|f(x)-f(y) \right|\leq M|x-y|^{\alpha}\,\,(x,y\in[a,b]).
\end{equation}

\begin{thm}
Let $q:=(q_{n})$ be a sequence satisfies the condition (\ref{y1}) and
$f \in Lip_{M}(\alpha)$, $0<\alpha\leq 1$, then we have
\begin{equation*}
|T_{n,t}^{\alpha,\beta}(f;q_n;.)-f(.)|\leq M\delta^{\alpha}_{n},
\end{equation*}
where  $\delta_{n}=(T_{n,t}^{\alpha ,\beta
}(e_{1}(t)-s)^{2},q;x)^{1/2}$.
\end{thm}

\begin{proof}
Since $T_{n,t}^{\alpha ,\beta}$ is linear and positive and by (\ref{a3}), we obtain
\begin{eqnarray*}
|T_{n,t}^{\alpha,\beta}(f;q_n;x)-f(x)|&\leq&
T_{n,t}^{\alpha,\beta}(|f(t)-f(x)|,q;x)\\
&\leq& M T_{n,t}^{\alpha,\beta}(|t-x|^{\alpha},q;x).
\end{eqnarray*}
Assuming $p=\frac{1}{\alpha}, q=\frac{\alpha}{2-\alpha}$ and applying
the H\"{o}lder  inequality, we get
\begin{eqnarray*}
|T_{n,t}^{\alpha,\beta}(f;q_n;x)-f(x)|&\leq&
T_{n,t}^{\alpha,\beta}(|f(t)-f(x)|,q;x)\\
&\leq& M \{T_{n,t}^{\alpha,\beta}(e_{1}-x)^{2},q;x)\}^{\alpha/2}.
\end{eqnarray*}
Taking $\delta_{n,t}= (T_{n,t}^{\alpha ,\beta
}(e_{1}-s)^{2},q;x)^{1/2}$. We get the desired result.
\end{proof}

\section{Local Approximation}
In this section, we state the local approximation theorem of the
operators $T_{n,t}^{\alpha ,\beta }\left(f;q;x \right)$. Let
$C_B\left[0,\infty \right)$ be the space of all real valued
continuous bounded functions $f$ on $\left[0,\infty \right)$ with
the norm $\parallel f \parallel=\sup\left\lbrace \vert f(x)\vert :
x\in[0,\infty)\right\rbrace $. The K-functional of $f$ is defined by
\begin{eqnarray*}
K_2(f;\delta)=\inf_{g\in W^2}\left\lbrace \Vert f-g\Vert +
\delta\Vert g'' \Vert\right\rbrace,
\end{eqnarray*}
where $\delta>0$ and $W^2=\left\lbrace g\in C_B\left[0,\infty \right): g', g''\in C_B[0,\infty) \right\rbrace $. By
Devore-Lorentz \cite[p. 177]{lorenz}, there exists an absolute
constant $C>0$ such that
\begin{eqnarray} \label{5.1}
K_2\left(f,\delta\right)\leqslant C \omega_2\left(f,\sqrt{\delta}\right)
\end{eqnarray}
where
\begin{eqnarray*}
\omega_2\left(f,\sqrt{\delta}\right)=\sup_{0<h\leq 0} \sup_{x\in
[0,\infty)}\left\vert f(x+2h)-2f(x+h)+f(x)\right\vert
\end{eqnarray*}
is the second order modulus of smoothness of $f$. Moreover,
\begin{eqnarray*}
\omega(f,\delta)=\sup_{0<h\leq 0} \sup_{x\in [0,\infty)}\vert
f(x+h)-f(x) \vert
\end{eqnarray*}
denotes the modulus of continuity of $f$.

Now, we give the direct local approximation theorem for the operators
$T_{n,t}^{\alpha ,\beta }\left(f;q;x \right)$.

\begin{thm}
Let $q \in (0,1)$. We have
\begin{eqnarray*}
\vert T_{n,t}^{\alpha ,\beta }\left(f;q;x \right)-f(x+t)\vert
\leqslant  C \omega_2\left(f,\delta_{n} \right)+
\omega\left(f,\frac{\alpha+D_{q}(A(1))R-\beta t}{[n]_q+\beta}\right)
\end{eqnarray*}
$\forall x\in $ $[0,\infty)$, $f\in C_B\left[0,\infty \right)$,
where $C$ is a positive constant.
\end{thm}

\begin{proof}
Let us define the following operators
\begin{eqnarray}\label{5.2}
\widetilde{T}_{n,t}^{\alpha ,\beta }\left(f;q;x \right)=
T_{n,t}^{\alpha ,\beta }\left(f;q;x \right)-f\left(
x+\frac{[n]_qt+\alpha+D_{q}(A(1))R}{[n]_q+\beta} \right)+f(x+t),
\end{eqnarray}
 $x\in [0,\infty)$. The operators $\widetilde{T}_{n,t}^{\alpha ,\beta }\left(f;q;x \right)$ are linear. Thus, we have the following:
 \begin{eqnarray}\label{5.3}
 \widetilde{T}_{n,t}^{\alpha ,\beta }\left(s-(x+t);q;x \right)=0,
 \end{eqnarray}
 (see Lemma  \ref{lemma3.1}).
Let $g\in W^2$, from Taylor's expansion
\begin{eqnarray*}
g(s)=g(x+t)+g'(x+t)\left(s-(x+t)\right)+ \int_{x+t}^s (s-u)g''(x)du,
\end{eqnarray*}
$s\in[0,\infty)$ and (\ref{5.3}) we obtain
\begin{eqnarray*}
\widetilde{T}_{n,t}^{\alpha,\beta }\left(g;q;x \right)=g(x+t)+
\widetilde{T}_{n,t}^{\alpha,\beta }\left(\int_{x+t}^s (s-u)g''(x)du
\right).
\end{eqnarray*}

By (\ref{5.2}), we have the following
\begin{eqnarray}\label{5.4}
\vert \widetilde{T}_{n,t}^{\alpha,\beta}\left(g;q;x\right)-g(x+t)
\vert &\leqslant & \left\vert T_{n,t}^{\alpha,\beta
}\left(\int_{x+t}^s (s-u)g''(u)du \right)\right\vert\nonumber\\ &&+
\left\vert
\int_{x+t}^{x+\frac{[n]_qt+\alpha+D_{q}(A(1))R}{[n]_q+\beta}}
\left( x+\frac{[n]_qt+\alpha+D_{q}(A(1))R}{[n]_q+\beta}-u\right)g''(u)du \right\vert \nonumber \\
&\leqslant & T_{n,t}^{\alpha,\beta} \left(\left\vert \int_{x+t}^s (s-u)g''(u)du  \right\vert ,x \right)+\int_{x+t}^{x+\frac{[n]_qt+\alpha+D_{q}(A(1))R}{[n]_q+\beta}} \left\vert  x+\frac{[n]_qt+\alpha+R}{[n]_q+\beta}-u \right\vert \left\vert g''(u) \right\vert \nonumber \\
&\leqslant & \left( T_{n,t}^{\alpha,\beta}\left(s-(x+t)\right)^2+
\left(x+\frac{[n]_qt+\alpha+D_{q}(A(1))R}{[n]_q+\beta} \right)^2
 \right)\Vert g'' \Vert
\end{eqnarray}

Using (\ref{a1}), we get
\begin{eqnarray*}
T_{n,t}^{\alpha,\beta}\left((s-(x+t))^2;q;x\right)+\left(x+\frac{[n]_qt+\alpha+D_{q}(A(1))R}{[n]_q+\beta}
\right)^2 \leqslant
\delta_{n}+\left(x+\frac{[n]_qt+\alpha+D_{q}(A(1))R}{[n]_q+\beta}
\right)^2.
\end{eqnarray*}
Thus, by (\ref{5.4}), we obtain
\begin{eqnarray}\label{5.5}
\vert \widetilde{T}_{n,t}^{\alpha,\beta}\left(g;q;x\right)-g(x+t)
\vert &\leqslant & \delta_{n}+
\left(x+\frac{[n]_qt+\alpha+D_{q}(A(1))R}{[n]_q+\beta} \right)^2.
\end{eqnarray}
By (\ref{A}), (\ref{lemma3.1}) and (\ref{5.2}), we get
\begin{eqnarray}\label{5.6}
\vert \widetilde{T}_{n,t}^{\alpha,\beta}\left(f;q;x\right) \vert &\leqslant & T_{n,t}^{\alpha,\beta}\left(f;q;x\right)+2\Vert f \Vert \nonumber \\
&\leqslant & \Vert f \Vert T_{n,t}^{\alpha,\beta}\left(1;q;x\right)+2\Vert f \Vert \nonumber \\
&\leqslant & 3\Vert f \Vert.
\end{eqnarray}

Now, by (\ref{5.2}), (\ref{5.5}) and (\ref{5.6})
\begin{eqnarray*}
\vert T_{n,t}^{\alpha,\beta}\left(f;q;x\right)-f(x+t) \vert &\leqslant &  \vert \widetilde{T}_{n,t}^{\alpha,\beta}\left(f-g;q;x\right)-(f-g)(x+t) \vert + \vert\widetilde{T}_{n,t}^{\alpha,\beta}\left(g;q;x\right)-g(x+t) \vert \\
&&+ \left\vert f\left(x+\frac{[n]_qt+\alpha+D_{q}(A(1))R}{[n]_q+\beta} \right)-f(x+t) \right\vert \\
&\leqslant & 4\Vert f-g \Vert + \delta_{n})\Vert g'' \Vert
\end{eqnarray*}
In view of (\ref{5.1}), $\forall q\in (0,1)$ we get
\begin{eqnarray*}
\vert T_{n,t}^{\alpha,\beta}\left(f;q;x\right)-f(x+t) \vert
\leqslant C\omega_2 \left(f,\delta_{n}\right) +\omega\left(
\frac{\alpha+D_{q}(A(1))R-\beta t}{[n]_q+\beta} \right)
\end{eqnarray*}
and this concludes the proof.
\end{proof}

\end{document}